\documentclass[12pt]{amsart}

\usepackage{amssymb, amsmath}

\usepackage{amsfonts, a4}
\usepackage{amsthm}
\usepackage{eufrak}

\newcommand{\goth}[1]{\EuFrak{#1}}
\newcommand{\C}{{\bf C}}
\newcommand{\R}{{\bf R}}
\newcommand{\Q}{{\bf Q}}
\newcommand{\Z}{{\bf Z}}
\newcommand{\Bsl}{{\bf SL}}

\newtheorem{thm}{Theorem}
\newtheorem{conj}{Conjecture}
\newtheorem{lem}{Lemma}[section]
\newtheorem{prop}{Proposition} %[section]

\begin{document}
%%%%%%%%%%%%%%%%%%%%%%%%%%%%%
\title[A non-homogeneous orbit closure of a diagonal subgroup]{A non-homogeneous orbit closure of a diagonal subgroup}

\author{Fran\c cois Maucourant} 
\address{Universit\'e Rennes I, IRMAR, Campus de Beaulieu 35042 Rennes cedex -  France}
\email{francois.maucourant@univ-rennes1.fr}

%%%%%%%%%%%%%%%%%%%%%%%%
%%%%%%%%%%%%%%%%%%%%%%%%

\begin{abstract}
 Let $G=\Bsl(n,\R)$ with $n\geq 6$. We construct examples of lattices $\Gamma \subset G$, subgroups $A$ of the diagonal group $D$ and points $x\in G/\Gamma$ such that the closure of the orbit $Ax$ is not homogeneous and such that the action of $A$ does not factor through the action of a one-parameter non-unipotent group. This contradicts a conjecture of Margulis.
 \end{abstract}
\maketitle

%%%%%%%%%%%%%%%%%%%%%%%%

\section{Introduction}

\subsection{Topological rigidity and related questions}

  Let $G$ be a real Lie group, $\Gamma$ a lattice in $G$, meaning a discrete subgroup of finite covolume, and $A$ a closed connected subgroup. We are interested in the action of $A$ on $G/\Gamma$ by left multiplication; we will restrict ourselves to the topological properties of these actions, referring the reader to \cite{kss} and  \cite{ekl}  for references and recent developments on related measure theoretical problems. 

  Two linked questions arise when one studies continuous actions of topological groups: what are the closed invariant sets, and what are the orbit closures?
   
  In the homogeneous action setting we are considering, there is a class of closed sets that admit a simple description: a closed subset $X\subset G$ is said to be {\it homogeneous} if there exists a closed connected subgroup $H\subset G$ such that $X=Hx$ for some (and hence every) $x\in X$. Let us say that the action of $A$ on $G/\Gamma$ is {\em topologically rigid} if for any $x\in G/\Gamma$, the closure $\overline{Ax}$ of the orbit $Ax$ is homogeneous. 
  
 The most basic example of a topologically rigid action is when $G=\R^n$, $\Gamma=\Z^n$, $A$ any vector subspace of $G$. It turns out that the behavior of elements of $A$ for the adjoint action on the Lie algebra $\goth{g}$ of $G$ plays a important role for our problem. Recall that an element $g\in G$ is said to be ${\bf Ad} $-unipotent if ${\bf  Ad}(g)$ is unipotent, and ${\bf Ad}$-split over $\R$ if ${\bf  Ad}(g)$ is diagonalizable over $\R$. If the closed, connected subgroup $A$ of $G$ is generated by ${\bf Ad}$-unipotent elements, a celebrated theorem of Ratner \cite{rat} asserts that the action of $A$ is always topologically rigid, settling a conjecture due to Raghunathan.
 
 When $A$ is  generated by elements which are ${\bf Ad}$-split over $\R$, much less is known.
 Consider the model case of $G=\Bsl(n,\R)$ and $A$ the group of diagonal matrices with nonnegative entries. If $n=2$, it is easy to produce non-homogeneous orbit closures (see e.g. \cite{lw}); more generally, a similar phenomenon can be observed when $A$ is a one-parameter subgroup of the diagonal group (see \cite{kss}, 4.1). However, for $A$ the full diagonal group, if $n \geq 3$, to the best of our knowledge, the only nontrivial example of a nonhomogeneous $A$-orbit closure is due to Rees, later generalized in \cite{lw}. In an unpublished preprint, Rees exhibited a lattice $\Gamma$ of $G=\Bsl(3,\R)$ and a point $x \in G/\Gamma$ such that for the full diagonal group $A$, the orbit closure $\overline{Ax}$ is not homogeneous. Her construction was based on the following property of the lattice: there exists a $\gamma \in \Gamma\cap A$ such that the centralizer $C_G(\gamma)$ of $\gamma$ is isomorphic to $\Bsl(2,\R) \times \R^*$, and such that $C_G(\gamma)\cap \Gamma$ is, in this product decomposition and up to finite index, $\Gamma_0 \times \langle \gamma \rangle$, where $\Gamma_0$ is a lattice in $\Bsl(2,\R)$ (see \cite{df}, \cite{lw}). Thus in this case the action of $A$ on $C_G(\gamma)/C_G(\gamma)\cap \Gamma$ factors to the action of a 1-parameter non-unipotent subgroup on $\Bsl(2,\R)/\Gamma_0$, which, as we saw, has many non-homogeneous orbits.
 
 Rees' example shows that factor actions of $1$-parameter non-${\bf Ad}$-unipotent  groups are obstructions to the topological rigidity of the action of diagonal subgroups. The following conjecture of Margulis \cite[conjecture 1.1]{mar} (see also \cite[4.4.11]{kss}) essentially states that these are the only ones:

\begin{conj} \label{margulis}
  Let $G$ be a connected Lie group, $\Gamma$ a lattice in $G$, and $A$
  a closed, connected subgroup of $G$ generated by ${\bf Ad}$-split
  over $\R$ elements. Then for any $x\in G/\Gamma$, one of the
  following holds :
\begin{itemize}
\item[(a)]$\overline{Ax}$ is homogeneous, or
\item[(b)] There exists a closed connected subgroup $F$ of $G$ and a
  continuous epimorphism $\phi$ of $F$ onto a Lie group $L$ such that
 \begin{itemize}
 \item $A\subset F$, 
 \item $Fx$ is closed in $G/\Gamma$ , 
 \item $\phi(F_x)$ is closed in $L$, where $F_x$ denotes the
   stabilizer $\{ g \in F | gx=x \}$,
 \item $\phi(A)$ is a one-parameter subgroup of $L$ containing no
   nontrivial ${\bf Ad}_L$-unipotent elements.
\end{itemize}
\end{itemize}
\end{conj}
 
 A first step toward this conjecture has been done by Lindenstrauss and Weiss \cite{lw}, who proved that in the case $G=\Bsl(n,\R)$ and $A$ the full diagonal group, if the closure of a $A$-orbit contains a compact $A$-orbit that satisfy some irrationality conditions, then this closure is homogeneous. See also \cite{tom}. Recently, using an approach based on measure theory, Einsiedler, Katok and Lindenstrauss proved that if moreover $\Gamma=\Bsl(n,\Z)$, then the set of bounded $A$-orbits has Hausdorff dimension $n-1$ \cite[Theorem 10.2]{ekl}.
 
\subsection{Statement of the results} 

 In this article we exhibit some counterexamples to the above conjecture when $G=\Bsl(n,\R)$ for $n\geq 6$ and $A$ is some strict subgroup of the diagonal group of matrices with nonnegative entries. 
 Let $D$ be the diagonal subgroup of $G$; note that $D$ has dimension $n-1$. Our main result is:

\begin{thm} \label{maintheorem} Assume $n\geq 6$.
  \begin{enumerate}
  \item There exists a $(n-3)$ dimensional closed and
      connected subgroup $A$ of $D$, and a point $x \in
    \Bsl(n,\R)/\Bsl(n,\Z)$ such that the closure of the $A$-orbit of
    $x$ satisfies neither condition {\rm (a)} nor condition {\rm (b)}
    of the conjecture.
  \item There exists a lattice $\Gamma$ of $\Bsl(n,\R)$, a
    $(n-2)$ dimensional closed and connected subgroup $A$ of
    $D$ and a point $x \in \Bsl(n,\R)/\Gamma$ such that the closure of
    the $A$-orbit of $x$ satisfies neither condition {\rm (a)} nor
    condition {\rm (b)} of the conjecture.
  \end{enumerate}
\end{thm}

  It will be clear from the proofs that these examples however satisfy a third condition:
\begin{itemize}
\item[(c)] 
{\em There exists a closed connected subgroup $F$ of $G$ and two
  continuous epimorphisms $\phi_1,\phi_2$ of $F$ onto Lie groups
  $L_1,L_2$ such that
 \begin{itemize}
 \item $A\subset F$, 
 \item $Fx$ is closed in $G/\Gamma$ , 
 \item For $i=1,2$, $\phi_i(F_x)$ is closed in $L_i$, 
 \item $(\phi_1,\phi_2): F \rightarrow L_1 \times L_2$ is surjective
 \item $(\phi_1,\phi_2): A\rightarrow \phi_1(A)\times \phi_2(A)$ is
   not surjective.
\end{itemize}
}
\end{itemize}

 Construction of these examples is the subject of Section \ref{sketch}, whereas the proof that they satisfy the required properties is postponed to Section \ref{sectionproof}.
      
  \subsection{Toral endomorphisms}
 
  To conclude this introduction, we would like to mention that the idea behind this construction can be also used to yield examples of 'non-homogeneous' orbits for diagonal toral endomorphisms. 
  
  Let $1<p_1<\dots < p_q$, with $q \geq 2$, be integers generating a multiplicative non-lacunary semigroup of $\Z$ (that is, the $\Q$-subspace  $\oplus_{1\leq i \leq q} \Q \log(p_i)$ has dimension at least $2$). We consider the abelian semigroup $\Omega$ of endomorphisms of the torus $T^n=\R^n/\Z^n$ generated by the maps $z \mapsto p_i z\,  \mathrm{mod} \, \Z^n$, $1\leq i \leq q$.
  
   In the one-dimensional situation, described by Furstenberg \cite{fur}, every $\Omega$-orbit is finite or dense. If $n\geq 2$, Berend \cite{ber} showed that minimal sets are the finite orbits of rational points, but there are others obvious closed $\Omega$-invariant sets, namely the orbits of rational affine subspaces.   
  Meiri and Peres \cite{MP} showed that closed invariant sets have integer Hausdorff dimension. 
  
   Note that the study of the orbit of a point lying in a proper rational affine subspace reduces to the study of finitely many orbits in lower dimensional tori, although some care must be taken about the pre-periodic part of the rational affine subspace (for example, if $q=n=2$, and if $\alpha \in T^1$ is irrational with non-dense $p_1$-orbit, the orbit closure of the point $(\alpha,1/p_2)\in T^2$ is the union of a horizontal circle and a finite number of strict closed infinite subsets of some horizontal circles). 
   
   With this last example in mind, Question 5.2 of \cite{MP} can be re-formulated: is a proper closed invariant set necessarily a subset of a finite union of rational affine tori? Or, equivalently, if a point is outside any rational affine subspace, does it necessarily have a dense orbit? It turns out that this is not the case at least for $n\geq 2q$, as the following example shows. 
 
 \begin{thm} \label{toralexample}
 Let $N$ be an integer greater than $q\frac{\log  p_q}{\log p_1}$, and let $z$ be the point in the $2q$-dimensional torus $T^{2q}$ defined by the coordinates modulo 1:
 $$z=(z_1,\dots, z_{2q})=\left(  \sum_{k\geq 1} p_1^{-N^{2k}}, \dots,  \sum_{k\geq 1} p_q^{-N^{2k}}, \sum_{k\geq 1} p_1^{-N^{2k+1}}, \dots, \sum_{k\geq 1} p_q^{-N^{2k+1}} \right).$$
 Then the  point $z \in T^{2q}$ is not contained in any rational affine subspace, but its orbit $\Omega z$ is not dense. 
 \end{thm}
 
 The proof of Theorem \ref{toralexample} will be the subject of Section \ref{sectionproof2}.
 
\section{Sketch of proof of Theorem \ref{maintheorem}}  
\label{sketch}
 \subsection{The direct product setup}  
 We now describe how these examples are built. %Proofs are postponed to section \ref{sectionproof}.
 Choose two integers $n_1\geq 3$, $n_2\geq 3$, such that $n_1+n_2=n$.
 For $i=1,2$, let $\Gamma_i$ be a lattice in $G_i=\Bsl(n_i,\R)$. 
 
 Let $g_i$ be an element of $G_i$ such that $g_i \Gamma_i g_i^{-1}$ intersects the diagonal subgroup $D_i$ of $\Bsl(n_i,\R)$ in a lattice, in other words $g_i  \Gamma_i$ has a compact $D_i$-orbit; such elements exist, see \cite{PR}. In fact, we will need an additional assumption on $g_i$, namely that the tori $g_i^{-1} D_i g_i$ are {\em irreducible over $\Q$}. The precise definition of this property and the proof of the existence of such a $g_i$, a consequence of a theorem of Prasad and Rapinchuk \cite[Theorem 1]{PRR},  will be the subject of Section \ref{prasadrapinchuk}.
 
 Let $\pi_i : G_i \rightarrow G_i / \Gamma_i$ be the canonical quotient map. Define for $i=1,2$:
 $$y_i=\pi_i \left( 
   \left[ 
  \begin{array}{ccccc}
  1 & 0 & \dots 	& 0 & 1 \\
  0 & 1 & \ddots 	&   & 0\\
  \vdots & \ddots & \ddots & \ddots & \vdots \\
  \vdots &  & \ddots 		& 1 & 0 \\
  0 & \dots & \dots 		& 0 & 1 \\
  \end{array}
  \right]
  g_i \right).$$
  
  The $D_i$-orbit of $y_i$ is dense, by the following argument. It is easily seen that the closure of $D_iy_i$ contains the compact $D_i$-orbit $\mathcal{T}_i=\pi_i(D_ig_i)$. The $\Q$-irreducibility of $\mathcal{T}_i$ is sufficient to show that the assumptions of the theorem of Lindenstrauss and Weiss \cite[Theorem 1.1]{lw} are satisfied (Lemma \ref{noncompact}); thus, by this theorem, we obtain that there exists a group $H_i<G_i$ such that $H_iy_i=\overline{D_iy_i}$. Again because of $\Q$-irreducibility, the group $H_i$ is necessarily the full group, i.e. $H_i=G_i$ (proof of Lemma \ref{where}) 
  \footnote{The reader only interested in the case $n=6$ and $\Gamma=\Bsl(6,\Z)$ might note that when $\Gamma_1=\Gamma_2=\Bsl(3,\Z)$,  \cite[Corollary 1.4]{lw} can be used directly in the proof of Lemma \ref{where}; then the notion of $\Q$-irreducibility becomes unnecessary, and the entire Section \ref{prasadrapinchuk} can be skipped.}.
  
Let $A_1$ be  the $(n-3)$ dimensional subgroup of $G_1\times G_2$ given by:
\begin{equation} \label{definitionofA1}
\begin{split}
 A_1= \bigg\{ & \left( diag(a_1,..,a_{n_1}),diag(b_1,..,b_{n_2}) \right) \; : \\ 
 &  \;  \prod_{i=1}^{n_1} a_i=\prod_{j=1}^{n_2}b_j=\frac{a_1b_1}{a_{n_1}b_{n_2}}=1, \, a_i>0, b_j>0 \bigg\}.
\end{split}
\end{equation} 
   
 Then the $A_1$-orbit of $(y_1,y_2)$ is not dense in $G_1\times G_2/\Gamma_1\times \Gamma_2$ (Lemma \ref{notdense}), but $G_1\times G_2$ is the smallest closed connected subgroup $F$ of $G_1\times G_2$ such that $\overline{A_1(y_1,y_2)}\subset F(y_1,y_2)$ (Lemma \ref{smaller}). 
 
 This yields a counterexample to Conjecture \ref{margulis} which can be summarized as follows:
 
  \begin{prop} \label{intermediateresult}
  For $i=1,2$, let $n_i\geq 3$ and $\Gamma_i$ be a lattice in $G_i=\Bsl(n_i,\R)$. For $A_1$, $y_1,y_2$ depicted as above, the $A_1$-orbit of $(y_1,y_2)$ in $G_1\times G_2/\Gamma_1\times \Gamma_2$ satisfies neither condition (a) nor condition (b) of Conjecture \ref{margulis}. 
  \end{prop} 
   
\subsection{Theorem \ref{maintheorem}, part (1)} \label{firstpart}
In order to obtain the first part of Theorem \ref{maintheorem}, choose $\Gamma_i=\Bsl(n_i,\Z)$, $\Gamma=\Bsl(n,\Z)$ and consider the embedding of $G_1\times G_2$ in $G$, where matrices are written in blocks:
 \begin{equation} \label{defpsi}
 \Psi \, : \, (M_{n_1,n_1},N_{n_2,n_2})\mapsto \left[ \begin{array}{cc}M_{n_1,n_1} & 0_{n_1,n_2} \\ 
 0_{n_2,n_1} & N_{n_2,n_2}  \end{array}\right].
\end{equation}

 This embedding gives rise to an embedding $\overline{\Psi}$ of $G_1\times G_2/\Gamma_1\times \Gamma_2$ into $G/\Gamma$. Let $y_1,y_2$ be two points as above, let $x=\overline{\Psi}(y_1,y_2)$ and take $A=\Psi(A_1)$. We claim that this point $x$ and this group $A$ satisfy Theorem \ref{maintheorem}, part (1). In fact, since the image of $\overline{\Psi}$ is a closed connected $A$-invariant subset of $\Bsl(n,\R)/\Bsl(n,\Z)$, everything takes place in this direct product.
  
\subsection{Theorem \ref{maintheorem}, part (2)}  \label{secondpart}
The second part of Theorem \ref{maintheorem} is obtained as follows. Let $\sigma$ be the nontrivial field automorphism of the quadratic extension $\Q(\sqrt[4]{2})/\Q(\sqrt{2})$. Consider for any $m\geq 1$:
$${\bf SU}(m, \Z[\sqrt[4]{2}],\sigma)= \left\{ M \in \Bsl(m,\Z[\sqrt[4]{2}]) \, : \, (^tM^\sigma) M = I_m \right\}.$$ 
Then ${\bf SU}(m, \Z[\sqrt[4]{2}],\sigma)$ is a lattice in $\Bsl(m, \R)$, as will be proved in Section \ref{thelattice} (see \cite[Appendix]{df} for $m=3$). Define for $i=1,2$, 
$\Gamma_i={\bf SU}(n_i, \Z[\sqrt[4]{2}],\sigma)$, and $\Gamma={\bf SU}(n, \Z[\sqrt[4]{2}],\sigma)$. 
Now consider the map:
$$\varphi : G_1 \times G_2 \times \R \rightarrow G,$$
$$(X,Y,t)\mapsto \left[ \begin{array}{cc} e^{n_2 t} X & 0 \\ 
 0 & e^{-n_1 t}Y  \end{array}\right].$$
Define $M$ to be the image of $\varphi$. This time, $\varphi$ factors into a finite covering $\overline{\varphi}$ of homogeneous spaces:
$$\overline{\varphi}: G_1\times G_2 \times \R /\Gamma_1\times \Gamma_2 \times (\log \alpha) \Z  \rightarrow M/M\cap \Gamma \subset G/\Gamma,$$
where $\alpha=(3+2\sqrt{2})+\sqrt[4]{2}(2+2\sqrt{2})$ satisfies $\alpha^{-1}=\sigma(\alpha)$. 
 Consider the points $y_i$ constructed above, and let $x=\overline{\varphi}(y_1,y_2,0)$. Choose:
 $$A=\left\{ diag(a_1,..,a_n) \; | \; \prod_{i=1}^{n} a_i=\frac{a_1a_{n_1+1}}{a_{n_1}a_{n}}=1, \, a_i>0\right\}\subset \Bsl(n,\R).$$
 We claim that this lattice $\Gamma$,  this point $x$ and this group $A$ satisfy Theorem  \ref{maintheorem}, part (2). What happens here is that the $A$-orbit of $x$ is a circle bundle over an $A_1$-orbit (up to the finite cover $\overline{\varphi}$), like in Rees' example. 
 
\section{Proof of Theorem \ref{maintheorem}}
\label{sectionproof}
\subsection{{$\Q$}-irreducible tori}
\label{prasadrapinchuk}

 Fix $i\in \{1,2\}$. Recall that $\Gamma_i$ is a lattice in $G_i=\Bsl(n_i,\R)$. Since $n_i\geq 3$, by Margulis's arithmeticity Theorem \cite[Theorem 6.1.2]{Z}, there exists a semisimple algebraic $\Q$-group ${\bf H}_i$ and a surjective homomorphism $\theta$ from the connected component of identity of the real points of this group ${\bf H}_i^0(\R)$ to $\Bsl(n_i,\R)$, with compact kernel, such that $\theta({\bf H}_i(\Z)\cap {\bf H}_i^0(\R))$ is commensurable with $\Gamma_i$.
 
 Following Prasad and Rapinchuk, we say that a $\Q$-torus ${\bf T}\subset {\bf H}_i$ is $\Q$-irreducible if it does not contain any proper subtorus defined over $\Q$. By \cite[Theorem 1,(ii)]{PRR}, there exists a maximal $\Q$-anisotropic $\Q$-torus ${\bf T}_i\subset {\bf H}_i$, which is $\Q$-irreducible. Because any two maximal $\R$-tori of $\Bsl(n_i,\R)$ are $\R$-conjugate, there exists $g_i\in G_i$ such that $\theta({\bf T}_i^0(\R))=g_i^{-1}D_i g_i$. The subgroup ${\bf T}_i(\Z)$ is a cocompact lattice in ${\bf T}_i(\R)$ since ${\bf T}_i$ is $\Q$-anisotropic \cite[Theorem 8.4 and Definition 10.5]{B1}. Because $\theta({\bf H}_i(\Z)\cap {\bf H}_i^0(\R))$ and 
 $\Gamma_i$ are commensurable and $\theta$ has compact kernel, it follows that 
both $\Gamma_i \cap g_i^{-1}D_i g_i$ and $\theta({\bf T}^0_i(\Z)) \cap \Gamma_i \cap g_i^{-1}D_i g_i$ are also cocompact lattices in $g_i^{-1}D_i g_i$. The resulting topological torus $\pi_i (D_i g_i)  \subset G_i/\Gamma_i$ will be denoted ${\mathcal T}_i$. Write $z_i=\pi_i(g_i)$, so that ${\mathcal T}_i=D_i z_i$.

For every $1\leq k  < l \leq n_i$, define as in \cite{lw}: 
 $$N_{k,l}^{(i)}=\left\{ diag(a_1,.., a_{n_i}) \, : \, \prod_{s=1}^{n_i} a_s=1, \, a_k=a_l, \, a_s>0 \right\}\subset D_i,$$ 
 Of interest to us amongst the consequences of $\Q$-irreducibility is the fact that an element of $\Gamma_i \cap g_i^{-1}D_i g_i$ lying in a wall of a Weyl chamber is necessarily trivial. This is expressed in the following form:
  
 \begin{lem} \label{noncompact}
  For every $1\leq k  < l \leq n_i$, and any closed connected subgroup $L$ of positive dimension of $N_{k,l}^{(i)}$, the $L$-orbit of $z_i$ is not compact. 
 \end{lem} 
\begin{proof} Assume the contrary, that is $L z_i$ is compact. This implies that $g_i^{-1} L g_i \cap \Gamma_i$ is a uniform lattice in 
$g_i^{-1} L g_i$, so  $g_i^{-1} L g_i \cap \theta({\bf H}_i(\Z))$ is also a uniform lattice.
Since $L$ is nontrivial, there exists an element $\gamma \in {\bf H}_i(\Z)\cap {\bf H}_i^0(\R)$ of infinite order, such that $g_i \theta(\gamma) g_i^{-1}$ is in $L$. Note that since $\theta$ has compact kernel, ${\bf T}_i(\Z)$ is a lattice in $\theta^{-1} (\theta ({\bf T}_i^0(\R)))$ and is then a subgroup of finite index in ${\bf H}_i(\Z)\cap {\bf H}_i^0(\R)\cap \theta^{-1} (\theta ({\bf T}_i^0(\R)))$, so there exists $n>0$ such that $\gamma^n$ belongs to ${\bf T}_i(\Z)$.  
 Consider the representation:
\[
\begin{split}
\rho : \, & {\bf H}_i^0 (\R) \rightarrow {\bf GL}(\goth{sl}(n_i, \R)), \\
  & x \mapsto {\bf Ad}(g_i \theta(x) g_i^{-1}).
 \end{split}
 \]
 
 Recall that $\chi(diag(a_1,..,a_{n_1}))=a_k/a_l$ is a weight of ${\bf Ad}$ with respect to $D_i$, so 
 $\chi$ is a weight of $\rho$ with respect to ${\bf T}_i$. By \cite[Proposition 1, (iii)]{PRR}, the $\Q$-irreducibility of ${\bf T}_i$ implies that $\chi(\gamma^n)\neq 1$, but this contradicts the fact that $\theta(\gamma^n) \in g_i^{-1} N_{k,l}^{(i)} g_i$. 
\end{proof}

\subsection{Contraction and expansion} 
 
 For real $s$, denote by $a_i(s)$ the following $n_i\times n_i$-matrix:
 $$a_i(s)=diag(e^{s/2}, \underbrace{1, \dots, 1}_{n_i-2 \; \mathrm{times}},e^{-s/2}),$$
 and write simply $N_i$ for $N_{1,n_i}^{(i)}$. Write also:
   $$h_i(t)= 
      \left[ 
  \begin{array}{ccccc}
  1 & 0 & \dots 	& 0 & t \\
  0 & 1 & \ddots 	&   & 0\\
  \vdots & \ddots & \ddots & \ddots & \vdots \\
  \vdots &  & \ddots 		& 1 & 0 \\
  0 & \dots & \dots 		& 0 & 1 \\
  \end{array}
  \right].$$
 Then the following commutation relation holds:
 $$a_i(s) h_i(t)=h_i(e^{s}t)a_i(s), $$
 that is the direction $h_i$ is expanded for positive $s$; note that both $h_i$ and $a_i$ commute with elements of $N_i$. It is easy to check from Equation (\ref{definitionofA1}) that
 $$A_1=\left\{ (a_1(s)d_1,a_2(-s)d_2) \; : \; s \in \R, \; d_i\in N_i, \; i=1,2  \right\}.$$
 
 Recall that $y_i=h_i(1)z_i$. 
 
 \begin{lem} \label{where}
 
 \begin{enumerate}
 \item
 If $s\leq 0$, for any $d \in N_i$ the point $a_i(s)d y_i$ lies in the compact set $K_i=h_i([0,1]){\mathcal T}_i$. 
 \item The $D_i$-orbit of $y_i$ is dense in $G_i/\Gamma_i$.
 \item The set $\{a_i(s) d y_i \; : \; s\geq 0, \, d \in N_i \} $ is dense in $G_i/\Gamma_i$.
 \end{enumerate}
 \end{lem}
 \begin{proof} The first statement is clear from the commutation relation. It also implies that $D_i y_i$ contains the compact torus ${\mathcal T}_i$ in its closure.
 
  To prove the second point, we rely heavily on the paper of Lindenstrauss and Weiss. \cite[Theorem 1.1] {lw} applies here, since
 the hypotheses of their Theorem is precisely the conclusion of Lemma \ref{noncompact} for $L=N_{k,l}^{(i)}$. So the following holds: there exists a reductive subgroup $H_i$, containing $D_i$, such that $\overline{D_i y_i}=H_i y_i$, and $H_i\cap \Gamma_i$ is a lattice in $H_i$. Write $L=D_i \cap C_{G_i}(H_i)$.
 
  Since $D_i y_i$ is not closed, $H_i \neq D_i$, so there exists a nontrivial root relatively to $D_i$ for the Adjoint representation of $H_i$ on its Lie algebra, which is a subalgebra of $\goth{sl}(n_i,\R)$. Thus there exist $k,l$ such that $L\subset N_{k,l}^{(i)}$. By \cite[step 4.1 of Lemma 4.2]{lw}, $L z_i$ is compact, so by Lemma \ref{noncompact}, $L$ is trivial. By \cite[Proposition 3.1]{lw}, $H_i$ is the connected component of the identity of $C_{G_i}(L)$, so $H_i=G_i$, as desired.
 
  The third claim follows from the first and second claim together with the fact that $K_i$ has empty interior.
 \end{proof}
  
  \subsection{Topological properties of the $A_1$-orbit.}  
  
  \begin{lem} \label{notdense}
  The $A_1$-orbit of $(y_1,y_2)$  is not dense in $G_1\times G_2/\Gamma_1\times \Gamma_2$.
  \end{lem}
\begin{proof} 
 Consider the open set $U=K_1^c\times K_2^c$. We claim that the $A_1$-orbit of $(y_1,y_2)$ does not intersect $U$. Indeed, if $(a_1(s)d_1,a_2(-s)d_2)\in A_1$ with $s\in \R$ and $d_i \in N_i$, the previous Lemma implies that if $s \geq 0$, $a_2(-s)d_2y_2 \in K_2$, and if $s \leq 0$, $a_1(s)d_1y_1 \in K_1$. \end{proof}

 The following elementary result will be useful:
 \begin{lem} \label{elementary}
  Let $p_i: G_1\times G_2 \rightarrow G_i$ be the first (resp.\ second) coordinate morphism. If $F\subset G_1\times G_2$ is a subgroup such that $p_i(F)=G_i$ for $i=1,2$, and $A_1\subset F$, then $F=G_1\times G_2$.
  \end{lem} 
 \begin{proof}
  Let $F_1=Ker(p_1)\cap F$. Since $F_1$ is normal in $F$, $p_2(F_1)$ is normal in $p_2(F)=G_2$. 
  Note that $N_2 \subset p_2(A_1\cap Ker(p_1)) \subset p_2(F_1)$ is not finite, and that $G_2$ is almost simple, consequently the normal subgroup $p_2(F_1)$ of $G_2$ is equal to $G_2$. Let $(a,b) \in G_1\times G_2$, by assumption there exists $f \in F$ such that  $p_1(f)=a$. Let $f_1 \in F_1$ be such that $p_2(f_1)= b p_2(f)^{-1}$, then $(a,b)=f_1 f \in F$.
 \end{proof}

 We will have to apply several times the two following well-known Lemmas:

\begin{lem} \label{intersect}
Let $L$ be a Lie group, $\Lambda \subset L$ a lattice, $M,N$ two closed, connected subgroups of $L$, such that for some $w \in  L/\Lambda$, $Mw$ and $Nw$ are closed. Then $(M\cap N)w$ is closed.
\end{lem}
\begin{proof} This is a weaker form of \cite[Lemma 2.2]{S}.
\end{proof}

\begin{lem} \label{subgroup}
Let $L$ be a connected Lie group, $\Lambda \subset L$ a discrete subgroup, $M,N$ two subgroups of $L$, such that $M$ is closed and connected, and $N$ is a countable union of closed sets. For any $w \in L/\Lambda$, if $Mw \subset Nw$, then $M\subset N$.
\end{lem}
\begin{proof} Up to changing $\Lambda$ by one of its conjugate in $L$, one can assume that $w=\Lambda \in L/\Lambda$. By assumption, $M\Lambda \subset N\Lambda$ so $M \subset N\Lambda \subset L$. Recall that $M$ is closed, that $\Lambda$ is countable, and that $N$ is a countable union of closed sets, so Baire's category Theorem applies, and there exists $\lambda \in \Lambda$ and an open set $U$ of $M$ such that
$U \subset N\lambda$, so $UU^{-1} \subset N$. Since $M$ is a connected subgroup, $UU^{-1}$ generates $M$, so $M \subset N$. 
\end{proof}

  The following lemma will be useful both for proving that the closure of $A_1(y_1,y_2)$ is not homogeneous, and for proving it does not fiber over a $1$-parameter group orbit.
  
 \begin{lem} \label{smaller}
  Let $F$ be a closed connected subgroup of $G_1\times G_2$ such  that $F(y_1,y_2)$ contains the closure of $A_1(y_1,y_2)$. Then $F=G_1\times G_2$.
 \end{lem}
 \begin{proof}
  By Lemma \ref{where}, the set of first coordinates of  the set
  $$\{(a(s)d_1y_1,a(-s)d_2y_2): s \geq 0, d_i \in N_i \}, $$
 is dense in $G_1/\Gamma_1$ and the second coordinates lies in the compact set $K_2$, so the closure of $A_1(y_1,y_2)$ contains points of arbitrary first coordinate with their second coordinate in $K_2$. 
 Consequently, the set of first coordinates of $F(y_1,y_2)$ is the whole $G_1/\Gamma_1$, and similarly for the set of second coordinates. For $i=1,2$, Lemma \ref{subgroup} now applies to $L=M=G_i$, $\Lambda=\Gamma_i$, $N=p_i(F)$, which is a countable union of closed sets because $G_1\times G_2$ is $\sigma$-compact, and $w=y_i$, and so $p_i(F)=G_i$.
 
  In order to apply Lemma \ref{elementary} and finish the proof, we have to show that $A_1\subset F$. 
  Again, this follows from  a direct application of Lemma \ref{subgroup} to $L=G_1\times G_2$, $\Lambda=\Gamma_1 \times \Gamma_2$, $M=A_1$, $N=F$, $w=(y_1,y_2)$.
  
 \end{proof}
 
 \subsection{Proof of Theorem \ref{maintheorem}, part (1).} We now proceed to proving Theorem \ref{maintheorem}, part (1). The proof of Proposition \ref{intermediateresult} is similar and is omitted. 
 
 Recall that in this case, we fixed $A=\Psi(A_1)$ and $x=\overline{\Psi}(y_1,y_2)$.
 
 Assume $\overline{Ax}$ is homogeneous, that is $\overline{Ax}=Fx$ for a closed connected subgroup $F$ of $G$. Since $Ax \subset \overline{\Psi}(G_1\times G_2/\Gamma_1\times \Gamma_2)$, which is closed in $G/\Gamma$, Lemma \ref{subgroup} imply that $F \subset \Psi(G_1 \times G_2)$.
 By Lemma \ref{smaller}, $F=\Psi(G_1\times G_2)$, so $Fx=G/\Gamma$ and $Ax$ is dense in $\overline{\Psi}(G_1\times G_2)$, which is a contradiction.
 
 Now assume $\overline{Ax}$ fibers over the orbit of a one-parameter subgroup. Let $F$ be a closed connected subgroup, $L$ a Lie group and $\phi:F\rightarrow L$ a continuous epimorphism satisfying the (b) of the conjecture. Let $F'=F \cap \Psi(G_1\times G_2)$, we have $A \subset F'$. By Lemma \ref{intersect}, $F'x$ is closed in $Fx \cap \overline{\Psi}(G_1\times G_2)$, so is closed in $G/\Gamma$. By Lemma \ref{smaller} , $F'=\Psi(G_1\times G_2)$ necessarily. Let $H=Ker(\phi \circ \Psi)\subset G_1\times G_2$, so $A_1/(A_1\cap H)$ is a one-parameter group by assumption (b). 
 
 The subgroup $H$ is a normal subgroup of the semisimple group $G_1 \times G_2$, which has only four kind of normal subgroups : finite, $G_1\times G_2$, $G_1 \times \mathrm{finite}$ and $\mathrm{finite} \times G_2$. None of these possible normal subgroups have the property that they intersect $A_1$ in a codimension $1$ subgroup, so this is a contradiction. 

\subsection{The arithmetic lattice}
\label{thelattice}

 Here we prove that ${\bf SU}(n, \Z[\sqrt[4]{2}], \sigma)$ is a lattice in $\Bsl(n, \R)$. 
 Let $P,Q$ be the polynomials with coefficients in $\Q(\sqrt{2})$ such that for any $X,Y \in M_n(\C)$
$$\det(X+\sqrt[4]{2}Y)=P(X,Y)+\sqrt[4]{2}Q(X,Y).$$
 For an integral domain ${\bf A}\subset \C$, consider the set of pairs of matrices:
 \[
\begin{split}
 {\bf G}({\bf A})=\{ (X,Y)\in M_n({\bf A})^2 \, : \, & ^tX X - \sqrt{2} ^tYY=I_n, \, ^tXY- ^tYX=0, \\
 & \, P(X,Y)=1, \, Q(X,Y)= 0 \},  
\end{split}
\]
which implies that $(^tX-\sqrt[4]{2} ^tY)(X+\sqrt[4]{2}Y)=I_n$ and $\det(X+\sqrt[4]{2})=1$ for all $(X,Y) \in {\bf G}({\bf A})$.
Endow  ${\bf G}({\bf A})$ with the multiplication given by
$$(X,Y)(X',Y')=(XX' + \sqrt{2}YY', XY'+YX'),$$
which is such that the map $\phi: {\bf G}({\bf A})\rightarrow \Bsl(n,\C)$, $(X,Y)\mapsto X+\sqrt[4]{2}Y$ is a morphism. With this structure, ${\bf G}$ is an algebraic group, which is clearly defined over $\Q(\sqrt{2})$. Let $\tau$ be the nontrivial field automorphism of $\Q(\sqrt{2})/\Q$, 
it can be checked that the map $\phi$ is an  isomorphism between ${\bf G}({\bf R})$ and $\Bsl(n,\R)$, and that moreover $\phi': {\bf G^\tau}({\bf \R}) \rightarrow \Bsl(n,\C)$, $(X,Y)\mapsto X+i\sqrt[4]{2}Y$
is an isomorphism onto ${\bf SU}(n)$. Let ${\bf H}=Res_{\Q(\sqrt{2})/\Q}{\bf G}={\bf G}\times {\bf G}^\tau$. Then ${\bf H}$ is defined over $\Q$  (see for example \cite[6.1.3]{Z}, for definition and properties of the restriction of scalars functor). It follows from a Theorem of Borel and Harish-Chandra \cite[Theorem 3.1.7]{Z} that ${\bf H}(\Z)$ is a lattice in ${\bf H}(\R)$. Since ${\bf SU}(n)$ is compact, it follows that the projection of ${\bf H}(\Z)$ onto the first factor of ${\bf G}(\R) \times {\bf G}^\tau(\R)$ is again a lattice. Using the isomorphism between ${\bf G}(\R)$ and $\Bsl(n,\R)$, this projection can be identified with
$${\bf G}( \Z[\sqrt{2}])={\bf SU}(n,\Z[\sqrt{2}]+\sqrt[4]{2}\Z[\sqrt{2}],\sigma)={\bf SU}(n,\Z[\sqrt[4]{2}],\sigma).$$

\subsection{Proof of Theorem \ref{maintheorem}, part (2).} 

 Note that, as stated implicitely in Section \ref{secondpart}, 
 $$\varphi(\Gamma_1 \times \Gamma_2 \times (\log \alpha) \Z ) \subset \Gamma \cap M,$$
  so $\Gamma \cap M$ is a lattice in $M$, and $M/(M\cap \Gamma)$ is a closed, $A$-invariant subset of $G/\Gamma$. Notice also that the map $\Psi$ defined by Equation (\ref{defpsi}) defines an embedding $\overline{\Psi}: G_1\times G_2/\Gamma_1\times \Gamma_2 \rightarrow G/\Gamma$.

 Assume $\overline{Ax}$ is homogeneous, that is $\overline{Ax}=Fx$ for a closed connected subgroup $F$ of $G$. Since $Ax \subset M/(M\cap \Gamma)$, which is closed in $G/\Gamma$, Lemma \ref{subgroup} applied twice gives that $A \subset F \subset M$. Let $F'=F \cap \Psi(G_1\times G_2)$,  again by Lemma \ref{intersect}, $F'x$ is a closed subset of $Im(\overline{\Psi})$. Since $A_1 \subset F'$, $\Psi(A_1)x \subset F'x$ and Lemma \ref{smaller} implies that $F'=\Psi(G_1\times G_2)$. Since $A$ contains $\varphi(e,e,t)$ for all $t \in \R$, we have $M=AF' \subset F$ so $F=M$ necessarily.

 By Lemma \ref{notdense}, the $A_1$-orbit of $(y_1,y_2)$ is not dense; the topological transitivity of the action of $A_1$ on $G_1\times G_2/\Gamma_1\times \Gamma_2$ implies that moreover the closure of this orbit has empty interior. Thus, the $A_1\times \R$-orbit of $(y_1,y_2,0)$ is also nowhere dense in $G_1\times G_2 \times \R/\Gamma_1 \times \Gamma_2 \times (\log \alpha) \Z$. The map $\overline{\varphi}$ being a finite covering, the $A$-orbit of $x$ is nowhere dense. This is a contradiction with $F=M$.
 
   Now assume $\overline{Ax}$ fibers over the orbit of a one-parameter non-{\bf Ad}-unipotent subgroup. Let $F$ be a closed connected subgroup, $L$ a Lie group and $\phi:F\rightarrow L$ a continuous epimorphism satisfying the (b) of the conjecture. Let $F'=F \cap \Psi(G_1\times G_2)$ and $F''=F \cap M$, we have $A_1\subset F'$ and $A \subset F''$. Similarly, $F'x$ and $F''x$ are closed in $G/\Gamma$. Again, by Lemma \ref{smaller}, $F'=\Psi(G_1\times G_2)$ necessarily, and like before, $AF'\subset F'' \subset M$ so $F''=M$. 
   
   Let $H=Ker(\phi \circ \varphi)\subset G_1\times G_2\times \R$, so $A_1\times \R /(A_1\times \R \cap H)$ is a one-parameter group. This time, possibilities for the closed normal subgroup $H$ are: finite $\times \Lambda$, $G_1\times G_2\times \Lambda$, $G_1 \times \mathrm{finite} \times \Lambda$ and $\mathrm{finite} \times G_2 \times \Lambda$, where $\Lambda$ is a closed subgroup of $\R$. Of all these possibilities, only $G_1\times G_2 \times \Lambda$, where $\Lambda$ is discrete, has the required property that $A_1\times \R /(A_1\times \R \cap H)$ is a one-parameter group. This proves that $\Psi(G_1 \times G_2) \subset Ker(\phi)$, so $F \subset N_G(\Psi(G_1 \times G_2))$. However, the normalizer of $\Psi(G_1 \times G_2)$ in $G$ is the group of block matrices having for connected component of the identity the group $M$. So by connectedness of $F$, $F\subset M$, and since $M=F''\subset F$, we have $F=M$. Thus $L=F/Ker(\phi)=\R/\Lambda$ is abelian, and a fortiori every element of $L$ is unipotent; this contradicts (b).

\section{Proof of Theorem \ref{toralexample}}
\label{sectionproof2}

 The proof of Theorem \ref{toralexample} is divided in two independent lemmas.
 
 \begin{lem} The family $(z_1,\dots,z_{2q},1)$ is linearly independent over $\Q$.
 \end{lem}
\begin{proof}
 Consider a linear combination:
 $$\sum_{i=1}^{q} a_iz_i +b_iz_{i+q} = c.$$
 We can assume that $a_i,b_i$ and $c$ are integers. Let $k_0\geq 1$, write
 \begin{multline}\label{equationun}
 \left( \prod_{i=1}^q  p_i\right)^{N^{2k_0+1}} \left( \sum_{i=1}^{q} \sum_{k=1}^{k_0} a_i p_i^{-N^{2k}} + b_i p_i^{-N^{2k+1}} -c \right) =\\
 -\left( \prod_{i=1}^q  p_i\right)^{N^{2k_0+1}} 
 \left( \sum_{i=1}^{q} \sum_{k\geq k_0+1} a_i p_i^{-N^{2k}} + b_i p_i^{-N^{2k+1}}  \right) .
 \end{multline}
It is clear the left hand side is an integer. Since $1<p_1<\dots<p_q$, the right hand side is less in absolute value than
 \begin{multline*}
 p_q^{q N^{2k_0+1}}  2q \sup_i (|a_i|,|b_i|) \sum_{k \geq 0} \left( p_1^{-N^{2k_0+2}}\right)^{N^{2k}}\\
\leq  4q \sup_i (|a_i|,|b_i|) p_q^{q N^{2k_0+1}}  p_1^{-N^{2k_0+2}}\\
\leq 4q \sup_i (|a_i|,|b_i|) \exp( N^{2k_0+1} (q \log p_q - N \log p_1)).
\end{multline*}
Since $N>q\frac{\log(p_q)}{\log(p_1)}$, the last expression tends to zero . This proves the right-hand side of (\ref{equationun}) is zero for large enough $k_0$, so for all large $k$,
 $$\sum_{i=1}^{q} a_ip_i^{-N^{2k}} +b_i p_i^{-N^{2k+1}} = 0.$$
 The $p_i$ being distincts, this implies that for $i\in \{1,..,q\}$, $a_i=b_i=0$.
\end{proof}

 The following Lemma implies easily that the orbit of $z$ under $\Omega$ cannot be dense.

 \begin{lem} For all $\epsilon>0$, there exists $L>0$, such that for all $n_1,..,n_q \geq 0$ with $\sum_{i=1}^q n_i \geq L$, there exists $j \in \{1,\dots, 2q\}$ such that $p_1^{n_1}\cdots p_q^{n_q} z_j$ lies in the interval $[0,\epsilon]$ modulo $1$.
 \end{lem}
\begin{proof}
 Let $s \in \{1,\dots, q\}$ such that for all $r \in \{1,\dots, q\}$, $p_s^{n_s}\geq p_r^{n_r}$.
  Let $k_0$ be the integer part of $\log(n_s)/2\log(N)$, then either $N^{2k_0} \leq n_s \leq N^{2k_0+1}$, or $N^{2k_0+1}\leq n_s \leq N^{2k_0+2}$. In the first case, take $j=s$, then:
  $$p_1^{n_1} \cdots p_q^{n_q} z_j= p_1^{n_1} \cdots p_q^{n_q} \sum_{k \geq 1} p_s^{-N^{2k}}= p_1^{n_1} \cdots  p_q^{n_q}\sum_{k \geq k_0+1} p_s^{-N^{2k}} \, \mathrm{mod} \, 1.$$
 We have
  $$\sum_{k \geq k_0+1} p_s^{-N^{2k}} \leq 2 p_s^{-N^{2k_0+2}},$$
 so, using the fact that for all $r \in \{1,\dots,q\}$, $p_r^{n_r} \leq p_s^{n_s} \leq p_s^{N^{2k_0+1}}$, we obtain:
  $$p_1^{n_1} \cdots p_q^{n_q}\sum_{k \geq k_0+1} p_s^{-N^{2k}} 
  \leq 2 p_s^{qN^{2k_0+1}-N^{2k_0+2}}\leq 2 p_s^{N^{2k_0+1}(q-N)} ,$$
  but by hypothesis we have $N>q\frac{\log(p_q)}{\log(p_1)}>q$, so the preceding bound is small whenever $k_0$ is large. Because of the definition of $k_0$, we have
  $$k_0\geq \frac{\log \frac{\sum_{i=1}^q n_i \log p_i}{q\log p_q}}{2\log N}\geq \frac{\log \frac{L \log p_1}{q\log p_q}}{2 \log N},$$
  so $k_0$ is arbitrary large when $L$ is large.
  
  In the second case $N^{2k_0+1}\leq n_s \leq N^{2k_0+2}$, one can proceed similarly with $j=s+q$.
    
\end{proof}

\section{Acknowledgements}

 I am indebted to Yves Guivarc'h for mentioning to me the problem of orbit closure in the toral endomorphisms setting, and to Livio Flaminio for numerous useful comments. I also 
 thank Fran\c coise Dal'Bo, and S\'ebastien Gou\"ezel for stimulating conversations on this topic.
  
%%%%%%%%%%%%%%%%%%%%%%%%%%%%%%%%%%%%%%%%%%%%%%%%%%%%%%%  

\end{document}